\documentclass[11pt]{amsproc}
 \usepackage[margin=1in]{geometry}
\usepackage{setspace,fullpage}
\usepackage{graphicx}
\usepackage[nice]{nicefrac}
\usepackage{amssymb}
\usepackage{multirow}
\usepackage{array}

\usepackage{amsmath,amsthm,amscd,amssymb, mathrsfs}

\usepackage{latexsym}

\numberwithin{equation}{section}

\theoremstyle{plain}
\newtheorem{theorem}{Theorem}[section]
\newtheorem{lemma}[theorem]{Lemma}
\newtheorem{corollary}[theorem]{Corollary}

\newtheorem{conjecture}[theorem]{Conjecture}

\theoremstyle{definition}

\theoremstyle{remark}

\newtheorem{case[theorem]}{Case}

\title[\parbox{14cm}{\centering{Pinned distance problem over finite fields \hspace{1in}}} \quad]{Note on the pinned distance problem over finite fields  }
\author{Doowon Koh}
 
\address{Department of Mathematics\\
Chungbuk National University \\
Cheongju, Chungbuk 28644 Korea}
\email{koh131@chungbuk.ac.kr}

\thanks{
This work was conducted during the research year of Chungbuk National University in 2022, and  was supported by Basic Science Research Programs through National Research Foundation of Korea (NRF) funded by the Ministry of Education (NRF-2018R1D1A1B07044469).} 


\begin{document} 

\begin{abstract}
Let $\mathbb F_q$ be a finite field with odd $q$ elements. In this article, we prove that if $E\subseteq \mathbb F_q^d,  d\ge 2,$ and $|E|\ge  q$, then  there exists a set $Y \subseteq \mathbb F_q^d$ with $|Y|\sim q^d$ such that  for all $y\in Y$,  the number of distances between the point $y$ and the set $E$ is similar to the size of the finite field $\mathbb F_q.$  As a corollary,  we obtain that  for each set $E\subseteq \mathbb F_q^d$ with $|E|\ge q,$  there exists a set $Y\subseteq \mathbb F_q^d$ with $|Y|\sim q^d$ so that  any set  $E\cup \{y\}$ with $y\in Y$ determines a positive proportion of all possible distances. An averaging argument and the pigeonhole principle play a crucial role in proving our results.
\end{abstract}
\maketitle
\section{Introduction} 
Let $\mathbb F_q^d$ be the $d$-dimensional vector space over the finite field $\mathbb F_q$ with $q$ elements.
In 2005, Iosevich and Rudnev \cite{IR07} initially posed and studied an analogue of the Falconer distance problem over finite fields. 
They asked for the minimal exponent $\alpha>0$ such that  if $E\subseteq \mathbb F_q^d$ and $|E|\ge C q^\alpha$ for a sufficiently large constant $C>0$, then  
$$ |\Delta(E) | \ge c q$$
for some $0\le c\le 1,$ where $|\Delta(E)|$ denotes the cardinality of the distance set $\Delta(E)$, defined by
$$ \Delta(E)=\{||x-y||: x, y\in E\}.$$
Here we recall that   $||\alpha||:=\sum\limits_{j=1}^d \alpha_j^2$ for $\alpha=(\alpha_1, \ldots, \alpha_d) \in \mathbb F_q^d.$\\

By developing the discrete Fourier machinery, Iosevich and Rudnev \cite{IR07} proved that  $|\Delta(E)|\sim q$ whenever $|E|\ge C q^{(d+1)/2}.$ We recall that  $A \ll B$ means that  $A\le CB$ for some constant $C>0$, which is independent of $q$, and  we  use $A\sim B$   if $A\ll B $ and $B\ll A.$ The authors in \cite{HIKR10} showed that  the exponent $(d+1)/2$ is optimal for 
all odd dimensions $d\ge 3$ except for the cases when  $-1$ is not a square and $d=4k-1$ for $k\in \mathbb N.$
However, in any other cases including even dimensions $d\ge 2,$  it has been conjectured by Iosevich and Rudnev \cite{IR07} that  in order to have a positive proportion of all distances, the exponent $(d+1)/2$ can be improved to $d/2.$ 
\begin{conjecture} [Iosevich-Rudnev's Conjecture] Let $E\subseteq \mathbb F_q^d.$ 
Suppose that $d\ge 2$ is even or  $d, q \equiv 3 \mod{4}.$ Then if $|E|\ge C q^{d/2}$ for a sufficiently large constant $C>0$, we have
$|\Delta(E)|\sim q.$
\end{conjecture}

Iosevich-Rudnev's Conjecture is still open and  even  the threshold $(d+1)/2$ has not been improved except for two dimensions.
In the case of $d=2$ over general finite fields,  the authors in \cite{CEHIK09} obtained the $4/3$ exponent, which is  the first result to break down the exponent $(d+1)/2.$ This result was obtained by applying  the restriction estimates for the circles on the plane.
More precisely,  they proved the following result with an explicit constant.

\begin{theorem} [\cite{CEHIK09}]\label{wolffin2d}  Let $E$ be a subset of $\subset {\mathbb F}_q^2$ with $|E|\ge q^{4/3}.$ Then following statements hold:
\begin{enumerate}
\item If $q\equiv 3 \mod{4},$ then  $|\Delta(E)|\ge \frac{q}{1+\sqrt{3}}.$

\item If $q \equiv 1 \mod{4},$ then $|\Delta(E)|\ge C_{q} q,$ where the constant $C_q$ is defined by
$$C_q:= \frac{\left(1-2q^{-1}\right)^2}{1+\sqrt{3}-\sqrt{3}q^{-2/3}}.$$
\end{enumerate}
\end{theorem} 

Notice that  $C_q >0$ for all $q\ge 3,$ and  $C_q$ converges to  $\frac{1}{1+\sqrt{3}}$ as $q\to \infty.$  Since a convergent sequence is bounded,  we therefore choose  a constant $c>0,$ independent of $q,$ such that   $C_q \ge c >0.$  From this observation,  the following corollary is a direct consequence of Theorem \ref{wolffin2d}.
\begin{corollary}[\cite{CEHIK09}] Suppose that $E \subseteq \mathbb F_q^2$ with $|E|\ge q^{4/3}.$ Then we have
$$ |\Delta(E)|\sim q.$$
\end{corollary}

Using  a group action approach,  Bennett, Hart,  Iosevich,  Pakianathan, and  Rudnev \cite{BHIPR} provided  an alternative proof for the exponent $4/3$ in the above corollary.\\

As a strong version of the Falconer distance problem, one has  studied the pinned distance problem over finite fields.
Given $E \subseteq \mathbb F_q^d, d\ge 2,$ and $y\in \mathbb F_q,$ the pinned distance set with a pin $y$, denoted by $\Delta_y(E),$ is defined by
$$ \Delta_y(E)=\{||x-y||: x\in E\}.$$
The  Chapman,  Erdo\~{g}an, Hart,  Iosevich, and  Koh \cite{CEHIK09} showed that  the exponent $(d+1)/2$ due to Iosevich and Rudnev  holds true for the pinned distance sets.
More precisely they proved the following.
\begin{theorem}
[\cite{CEHIK09}] \label{CEHIK} Let $E\subseteq \mathbb F_q^d, d\ge 2.$  If $|E|\ge q^{\frac{d+1}{2}},$ then there exists  a subset $E'$ of $E$ with $|E'|\sim |E|$ so that  for every $y\in E'$,  we have
$$ |\Delta_y(E)|\sim q.$$
\end{theorem}

As seen in the conjecture of the Falconer distance set problem,    the exponent $(d+1)/2$ cannot be improved except for the cases when $d, q\equiv 3 \mod{4}$ or  $d\ge 2$ is even. 
However, in  those cases  it have been believed that   $d/2$ can be the best possible exponent for the pinned distance sets. 
As partial evidence for this prediction,    the $4/3$ exponent result was extended  to the pinned distance sets in $\mathbb F_q^2$  by Hanson, Lund, and  Roche-Newton \cite{HLR16}, who successfully performed  the bisector energy estimate.

\begin{theorem}[\cite{HLR16}] Let $E\subseteq \mathbb F_q^2.$  If $|E|\ge q^{4/3}$, then  the conclusion of Theorem \ref{CEHIK} holds.
\end{theorem}

When $q$ is prime,  the exponent $4/3$ have been improved to $5/4$ by Murphy,  Petridis,  Pham,  Rudnev, and  Stevenson \cite{MPPRS}.
\begin{theorem} [\cite{MPPRS}] Let $q$ be prime. Then if $E \subseteq \mathbb F_q^2$ with $|E|\ge q^{\frac{5}{4}}$,   we have
$$\max\limits_{y\in E} |\Delta_y(E)|\sim q.$$
\end{theorem}

Despite researchers' efforts, the conjectured exponent $d/2$ has not been proven. It is unlikely that one can establish the conjecture by using the known techniques.
 Moreover, there is very little evidence to support that the conjecture is true.\\

The main purpose of this paper is not to derive an improved  result on the distance problem, but to address that the probability that  random sets satisfy the distance conjecture is very high.

\subsection{The statement of main results}

Our main theorem is as follows.
\begin{theorem} \label{main} Let $E\subseteq \mathbb F_q^d.$ Then given $a>1,$ there exists $Y\subseteq \mathbb F_q^d$ with $|Y|\ge \frac{a-1}{a} q^d$ such that for all $y\in Y$,  
$$ |\Delta_y(E)|\ge \min\left\{\frac{q}{2a},~\frac{|E|}{2a}\right\}.$$
\end{theorem}

The following result is a direct consequence of Theorem \ref{main}.
\begin{corollary}
Suppose that $E\subseteq \mathbb F_q^d, d\ge 2,$ with $|E|\ge q.$ Then for any $a>1$, there exists $Y\subseteq \mathbb F_q^d$ with $|Y|\ge \frac{a-1}{a} q^d$ so that for all $y\in Y,$ we have
$$ |\Delta(E\cup \{y\})|\ge \frac{q}{2a}.$$ 
\end{corollary}
\begin{proof} Since $|E|\ge q,$  we have
$$ \min\left\{ \frac{q}{2a}, ~ \frac{|E|}{2a}\right\}=\frac{q}{2a}.$$
In addtition,  note that for all $y\in \mathbb F_q,$ we have $ |\Delta(E\cup \{y\})|\ge |\Delta_y(E)|.$
Hence, the statement of the corollary follows immediately from Theorem \ref{main}.
\end{proof}
\bibliographystyle{amsplain}

\begin{thebibliography}{10}



\bibitem{BHIPR}  M. Bennett, D. Hart, A. Iosevich, J. Pakianathan, and M. Rudnev, {\it Group actions and geometric combinatorics in $\mathbb F_q^d,$} Forum Math. \textbf{29} (2017),  91-110.

\bibitem{CEHIK09} J.~ Chapman, M.~ Erdo\~{g}an, D.~ Hart, A.~ Iosevich, and D.~ Koh, {\it Pinned distance sets, Wolff's exponent in finite fields and sum-product estimates}, Math.Z. \textbf{271}, (2012), 63-93

 
\bibitem{HIKR10} D.~ Hart, A.~ Iosevich, D.~ Koh and M.~ Rudnev, {\it Averages over hyperplanes, sum-product theory in vector spaces over finite fields and the Erd\"os-Falconer distance conjecture}, Trans. Amer. Math. Soc. Volume 363, Number 6,  (2011),  3255–3275.

\bibitem{HLR16} B. Hanson, B. Lund, and O. Roche-Newton, {\it On distinct perpendicular bisectors and pinned distances in finite fields,}  Finite Fields Appl. \textbf{37} (2016),  240-264.

\bibitem{IR07} A.~ Iosevich, M. ~Rudnev, {\it  Erd\H{o}s distance problem in vector spaces over finite fields}, Trans. Amer. Math. Soc. \textbf{359} (2007), 6127-6142.








 



 




\bibitem{MPPRS} B. Murphy, G. Petridis, T. Pham, M. Rudnev, and S. Stevenson, {\it On the pinned distances problem in positive characteristic,}  to appear in Journal of London Mathematical Society, 2021, arXiv:2003.00510.

\bibitem{MT04} G. Mockenhaupt, and T. Tao, \emph{Restriction and Kakeya phenomena for finite fields}, Duke Math. J. \textbf{121} (2004), 1, 35-74.

\bibitem{RS18} M. Rudnev and I. Shkredov, \emph{On the restriction problem for discrete paraboloid in lower dimension}, Adv. Math. {\bf 339} (2018), 657-671.
 


\end{thebibliography}

\section{Proof of main result (Theorem \ref{main})}
We begin with the standard counting argument as in \cite{CEHIK09}.

 To find a lower bound of the cardinality of the $y$-pinned distance set $\Delta_y(E)$,  we consider the $y$-pinned counting function $\nu_y: \mathbb F_q \to \mathbb N \cup \{0\},$ which maps  an element $t$ in $\mathbb F_q$ to the number of elements $x$ in $E$ such that $||x-y||=t.$ In other words,  for $y\in \mathbb F_q^d, t\in \mathbb F_q,$ we have
$$\nu_y(t)=\sum_{x\in E: ||x-y||=t} 1.$$

Since $|E|^2=\left(\sum_{t\in \Delta_y(E)} \nu_y(t)  \right)^2, $  it follows from the Cauchy-Schwarz inequality that 
\begin{equation}\label{PinForm} |\Delta_y(E)|\ge \frac{|E|^2} {\sum_{t\in \mathbb F_q} \nu_y^2(t)}.\end{equation}

\subsection{Key lemmas}
The average of $\sum_{t\in \mathbb F_q} \nu_y^2(t)$ over $y$ in $\mathbb F_q^d$  is explicitly given as follows:
\begin{lemma}\label{AVPin} Let $E\subseteq \mathbb F_q^d.$ Then we have
$$ \frac{1}{q^d} \sum_{y\in \mathbb F_q^d} \sum_{t\in \mathbb F_q} \nu_y^2(t) = \frac{|E|^2}{q} + \frac{q-1}{q} |E|.$$
\end{lemma}
\begin{proof}
By the definition of the $y$-pinned counting function $\nu_y(t)$,  we have for each $y\in \mathbb F_q,$ 
$$ \sum_{t\in \mathbb F_q} \nu_y^2(t) = \sum_{x,z\in E: ||x-y||=||z-y||} 1.$$
Hence,  the average of it over $y\in \mathbb F_q^d$  is given as follows:
\begin{align}\label{Sib} \frac{1}{q^d} \sum_{y\in \mathbb F_q^d} \sum_{t\in \mathbb F_q} \nu_y^2(t)&=\frac{1}{q^d} \sum_{x, z\in E: x=z} \sum_{y\in \mathbb F_q^d: ||x-y||=||z-y||} 1+\frac{1}{q^d} \sum_{x, z\in E: x\ne z} \sum_{y\in \mathbb F_q^d: ||x-y||=||z-y||} 1\\ \nonumber
&=|E| +\frac{1}{q^d} \sum_{x, z\in E: x\ne z} \sum_{y\in \mathbb F_q^d: ||x-y||=||z-y||} 1.\end{align}
\end{proof}
Now, we notice that for $x,z\in E$ with $x\ne z$, we have 
\begin{equation}\label{HPEq}\sum_{y\in \mathbb F_q^d: ||x-y||=||z-y||} 1 = q^{d-1}.\end{equation}
In fact, since $x\ne z$,  the quantity $\sum_{y\in \mathbb F_q^d: ||x-y||=||z-y||} 1$  is the number of the elements in the hyper-plane which bisects the line segment joining $x$ and $z$. Alternatively we can prove this rigorously  by using the finite field Fourier analysis. To see this,  let $\chi$ denote a nontrivial additive character of $\mathbb F_q.$ Then by the orthogonality of $\chi$,  we see that if  $x\ne z,$ then 
\begin{align*} \sum_{y\in \mathbb F_q^d: ||x-y||=||z-y||} 1
&= q^{-1} \sum_{y\in \mathbb F_q^d} \sum_{s\in \mathbb F_q}  \chi(s (||x-y||-||z-y||))\\
&=q^{d-1} + q^{-1} \sum_{y\in \mathbb F_q^d} \sum_{s\ne 0}  \chi(s (||x-y||-||z-y||)).\end{align*}
Applying the orthogonality of $\chi$ to the sum over $y$, we see that the second term above is zero since $\chi(s (||x-y||-||z-y||))= \chi(-2s (x-z)\cdot y)  \chi(s(||x||-||z||))$  and $s(x-z)$ is not a zero vector. Hence, the equation \eqref{HPEq} holds.

Finally, combining  the above two estimates \eqref{Sib}, \eqref{HPEq},  we obtain the desirable  estimate.\\

The following result can be obtained by the pigeonhole principle together with Lemma \ref{AVPin}.
\begin{lemma}\label{CorPinForm} Let $E\subseteq \mathbb F_q^d.$ Then for any $a>1,$ there exists $Y\subseteq \mathbb F_q^d$ with $|Y|\ge \frac{a-1}{a} q^d$  such that for every $y\in Y$,  
$$ \sum_{t\in \mathbb F_q} \nu_y^2(t) \le \frac{a}{q} |E|^2 + \frac{a(q-1)}{q} |E|.$$
\end{lemma}
\begin{proof}
Let us fix $a>1.$ Define 
$$Y=\left\{y\in \mathbb F_q^d: \sum_{t\in \mathbb F_q} \nu_y^2(t) \le \frac{a}{q} |E|^2 + \frac{a(q-1)}{q} |E|\right\}.$$
To complete the proof,  it remains to show that 
$$|Y|\ge \frac{a-1}{a} q^d.$$
By contradiction,  let us assume that 
\begin{equation}\label{False} |Y| < \frac{a-1}{a} q^d.\end{equation}
It is clear that 
\begin{equation}\label{Pige1} \mathbb F_q^d \setminus Y=\left\{y\in \mathbb F_q^d:  \sum_{t\in \mathbb F_q} \nu_y^2(t) > \frac{a}{q} |E|^2 + \frac{a(q-1)}{q} |E|\right\}.\end{equation}
We also notice that for all $y\in \mathbb F_q^d,$
\begin{equation} \label{Pige2} \sum_{t\in \mathbb F_q} \nu_y^2(t) \ge \sum_{t\in \mathbb F_q} \nu_y(t) = |E|.\end{equation}

Now by Lemma \ref{AVPin},  it follows that
\begin{equation}\label{LemA}\frac{1}{q^d} \sum_{y\in \mathbb F_q^d} \sum_{t\in \mathbb F_q} \nu_y^2(t) = \frac{|E|^2}{q} + \frac{q-1}{q} |E|.\end{equation}
However, we can also estimate it as follows.  Using \eqref{Pige1} and \eqref{Pige2},  we have
\begin{align*}\frac{1}{q^d} \sum_{y\in \mathbb F_q^d} \sum_{t\in \mathbb F_q} \nu_y^2(t) 
&= \frac{1}{q^d} \sum_{y\in  Y} \sum_{t\in \mathbb F_q} \nu_y^2(t) + \frac{1}{q^d} \sum_{y\in \mathbb F_q^d\setminus Y} \sum_{t\in \mathbb F_q} \nu_y^2(t)\\
& > \frac{1}{q^d} |Y| |E| + \frac{1}{q^d}(q^d-|Y|) \left(\frac{a}{q} |E|^2 + \frac{a(q-1)}{q} |E|\right)\\
& = \frac{a|E|^2}{q} + \frac{a(q-1)|E|}{q} + \left(\frac{|E|}{q^d}- \frac{a|E|^2}{q^{d+1}} -\frac{a(q-1)|E|}{q^{d+1}}\right) |Y|. \end{align*}
Since $a>1$,  in the third term above,  the coefficient of $|Y|$ is negative. Hence, we can combine the above estimate with  \eqref{False} to deduce that
$$ \frac{1}{q^d} \sum_{y\in \mathbb F_q^d} \sum_{t\in \mathbb F_q} \nu_y^2(t)
 > \frac{a|E|^2}{q} + \frac{a(q-1)|E|}{q} + \left(\frac{|E|}{q^d}- \frac{a|E|^2}{q^{d+1}} -\frac{a(q-1)|E|}{q^{d+1}}\right) \left(\frac{a-1}{a} q^d \right)$$
Simplifying the RHS of the above equation, we get
$$ \frac{1}{q^d} \sum_{y\in \mathbb F_q^d} \sum_{t\in \mathbb F_q} \nu_y^2(t) >\frac{|E|^2}{q} + \frac{q-1}{q} |E| + \frac{a-1}{a} |E|,$$
which contradicts the equation \eqref{LemA} since $a>1.$
\end{proof}

\subsection{Proof of Theorem \ref{main}}
Combining \eqref{PinForm} and Lemma \ref{CorPinForm}, we get the required result:
$$ |\Delta_y(E)|\ge \frac{|E|^2}{\frac{a}{q} |E|^2 + \frac{a(q-1)}{q} |E|} \ge \min\left\{ \frac{q}{2a}, ~ \frac{q|E|}{2a(q-1)} \right\}\ge \min\left\{ \frac{q}{2a}, ~ \frac{|E|}{2a}\right\}. $$

\end{document}